\documentclass[11pt,a4paper,twoside]{amsart}

\usepackage{amssymb}
\usepackage{graphics}
\usepackage{epic,eepic}

\newtheorem{theo}{\bf Theorem}[section]
\newtheorem{prop}[theo]{\bf Proposition}
\newtheorem{lemma}[theo]{\bf Lemma}
\newtheorem{defi}[theo]{\bf Definition}

\theoremstyle{remark}

\newcommand{\Z}{{\mathbb Z}}
\newcommand{\N}{{\mathbb N}}
\newcommand{\R}{{\mathbb R}}
\newcommand{\id}{{\rm id}}

\DeclareMathOperator{\sgn}{sgn}
\DeclareMathOperator{\Prob}{Prob}
\providecommand{\abs}[1]{\lvert#1\rvert}

\begin{document}

\title{Expected length of a product of random reflections}
\author{Jonas Sj{\"o}strand}
\address{Department of Mathematics, Royal Institute of Technology \\
  SE-100 44 Stockholm, Sweden}
\email{jonass@kth.se}

\date{23 November 2010}

\begin{abstract}
We present a simple formula for the expected number of inversions in a
permutation of size $n$ obtained by applying $t$ random (not necessarily
adjacent) transpositions to the identity permutation. More general,
for any finite irreducible Coxeter group belonging to one of
the infinite families
(type A, B, D, and I), an exact expression
is obtained for the expected length of a product of $t$ random
reflections.
\end{abstract}

\maketitle

\section{Introduction}
\noindent
In 2000, Eriksson et.~al.~\cite{erikssonerikssonsjostrand} studied the
expected number of inversions after $t$ random adjacent transpositions
applied to the identity permutation. Their main motivation came from
the area of molecular evolution where a genome (modelled as a permutation)
undergoes random mutations (modelled as adjacent transpositions). Since then
several people have analysed this and similar problems, motivated both by
biology and pure mathematics.

In 2005, Eriksen~\cite{eriksen} improved the results
from~\cite{erikssonerikssonsjostrand} and obtained an exact formula
for the expected number of inversions. More recently, in 2010,
Bousquet-M\'elou~\cite{bousquetmelou} gave a completely different
exact expression for the same thing! In 2002, Troili~\cite{troili}
generalized the question to other groups than the symmetric group,
and she gave corresponding exact formulas for the dihedral groups $I_2(m)$.

In 2004,
Eriksen and Hultman~\cite{eriksenhultman1} modelled the genome
rearrangement problem closer to reality by considering \emph{any} random
transpositions (and not just adjacent ones) as mutations. Instead of
the number of inversions they computed the expected \emph{absolute length}
of the resulting permutation, i.e.~the minimal
number of transpositions whose product is the permutation.
In a sequel~\cite{eriksenhultman2} they generalized their results to
a class of complex reflection groups.

For an expos\'e of the research on more general aspects of Markov chains
generated by random (adjacent) transpositions, such as mixing times,
we refer to Bousquet-M\'elou~\cite{bousquetmelou}.

In the present paper, we will \emph{generate} permutations in the same
way as Eriksen and
Hultman, i.e.~by random transpositions, but we will \emph{meausure} them
by the number of inversions as in the original papers. We will also
generalize the framework to all finite Coxeter groups and present exact
expectancy formulas for the groups $A_n$, $B_n$, $D_n$ and $I_2(m)$.
In Coxeter group lingo, we compute the expected \emph{length}
of a product of $t$ random \emph{reflections}.

The paper is organized as follows.
After a section with some basic definitions, we review the previously
known results mentioned above so that the reader gets a feeling for what kind
of expressions may appear for the expected (absolute) length and also what
type of
combinatorics is going on behind the formulas. Next, in
section~\ref{sec:typeI} we loosen up by computing the expected length
and absolute length of a product of random reflections or
simple reflections in the dihedral groups $I_2(m)$. After this relatively
simple task, in section~\ref{sec:typeABD}
we are ready to state and prove our main results: formulas for
the expected length of a product of random reflections in the Weyl
groups of type A, B, and D. Finally, in section~\ref{sec:future}
we summarize the current state of affairs and reveal some of our ideas for
future research.

\section{Definitions and notation}
\noindent
We will follow the notation from the book by Bj\"orner and
Brenti~\cite{bjornerbrenti} where also the definitions of the basic
concepts may be found. For the reader's convenience we have included
the most relevant definitions below.

A \emph{Coxeter group} is a group $W$
with a presentation of the form 
\[
\langle S\,|\,(ss')^{m(s,s')}=\id\text{ for all }s,s'\in S
\text{ with }m(s,s')\ne\infty\rangle,
\]
where $m$ is a function from $S$ to $\{1,2,\dotsc,\infty\}$
such that $m(s,s')=m(s',s)$ for all $s,s'\in S$ and $m(s,s')=1$ if and only if
$s=s'$. The elements
of $S$ (which we identify with their image in $W$ through the presentation)
are called \emph{generators} or \emph{simple reflections} and
the pair $(W,S)$ is called a \emph{Coxeter system}. The group elements
that are conjugate to some simple reflection are called \emph{reflections}
and the set of reflections is usually denoted by $T$.

By definition, any group element $w\in W$ is a (finite) product of simple
reflections, and the minimal $\ell$ such that $w$ is a product of $\ell$
simple reflections is called the \emph{length} of $w$ and
is denoted by $\ell(w)$. Similarly, the minimal $\ell'$ such that $w$
is a product of
$\ell'$ reflections is called the \emph{absolute length} of $w$ and is denoted
by $\ell'(w)$.

Now, let us make a very general definition.
\begin{defi}
Given a (multiplicative) monoid $M$, a nonempty finite subset $R\subseteq M$,
a ``length'' function $\varphi$ from $M$ to $\N=\{0,1,\dotsc\}$,
and an integer $t\ge0$, we define
\[
E^M_{R,\varphi}(t):=
\frac{1}{\abs{R}^t}
\sum_{r_1,\dotsc,r_t\in R}\varphi(r_1\dotsm r_t)
\]
to be the expected ``length'' of a product of $t$ random elements,
chosen independently and uniformly from $R$.
\end{defi}
As a special case, for a Coxeter system $(W,S)$ with $S$ finite,
$E^W_{S,\ell}(t)$ denotes the expected length of a product of
$t$ simple reflections. Similarly, if $W$ is finite,
$E^W_{T,\ell'}(t)$ denotes the expected absolute length of a product
of $t$ reflections.

As mentioned in the introduction,
exact formulas for $E^{W}_{S,\ell}(t)$ have been found
for the symmetric groups $(W,S)=A_n$ by Eriksen~\cite{eriksen}
and for the dihedral groups $(W,S)=I_2(m)$ by Troili~\cite{troili}.
Exact formulas for  $E^{W}_{T,\ell'}(t)$ have been found by
Eriksen and Hultman, first for the symmetric groups~\cite{eriksenhultman1}
and later for the complex reflection groups $G(r,1,n)$ where the
hyperoctahedral groups $B_n$ are included.

In the present paper we will compute $E^{W}_{T,\ell}$ for the
three infinite families of finite irreducible Coxeter systems
$A_n$, $B_n$, $D_n$, and $I_2(m)$.

\section{Known results}\label{sec:knownresults}
\noindent
To get a flavour of the kind of expressions that may appear, we will
begin by presenting the already known results and discussing
briefly what methods were used to obtain them, before we reveal
our own results in the next section.

\begin{theo}[Eriksen (2005), Bousquet-M\'elou (2010)]
Let $n\ge1$ and $t\ge0$.
The expected number of inversions after $t$ random adjacent transpositions
applied to the identity permutation in $A_n$ is
\[
E^{A_n}_{S,\ell}(t)=\sum_{r=1}^t\frac{1}{n^r}\binom{t}{r}
\sum_{s=1}^r\binom{r-1}{s-1}(-1)^{r-s}4^{r-s}g_{s,n},
\]
where
\[
g_{s,n}=\sum_{l=0}^n\sum_{k\in\N}(-1)^k(n-2l)
\binom{2\lceil s/2\rceil-1}{\lceil s/2\rceil+l+k(n+1)}
\sum_{j\in\Z} (-1)^j\binom{2\lceil s/2\rceil}{\lceil s/2\rceil+j(n+1)}.
\]
An alternative formula is
\[
E^{A_n}_{S,\ell}(t)=\frac{n(n+1)}{4}-\frac{1}{8(n+1)^2}
\sum_{k,j=0}^n
\frac{(\cos\alpha_j+\cos\alpha_k)^2}{\sin^2\!\alpha_j\sin^2\!\alpha_k}
\left(1-\frac{4}{n}(1-\cos\alpha_j\cos\alpha_k)\right)^t,
\]
where $\alpha_k=(2k+1)\pi/(2n+2)$.
\end{theo}
The proof of the first formula, due to Eriksen, is built on a
discrete heat flow process in a two-dimensional
lattice. His expression has a combinatorial interpretation in terms
of lattice walks. The second formula, due to Bousquet-M\'elou, was obtained
by attacking the same heat flow process by generating functions and the kernel
method. It has the advantage that the dependence on $t$ is very
explicit which makes it possible to analyse what happens in the limit when
$n$ and $t$ tend to infinity at the same time in various ways.

\begin{theo}[Troili, 2002]\label{th:troili}
Let $m\in\{2,3,\dotsc,\infty\}$ and let $t\ge0$.
Then the expected length of a product of $t$ random simple reflections
in the dihedral group $I_2(m)$ is
\[
E^{I_2(m)}_{S,\ell}(t)=\sum_{j=0}^{\lfloor\frac{t-1}{2}\rfloor}
\frac{1}{4^j}\left[\binom{2j}{j}+2\sum_{k\ge1}\binom{2j}{j-km}\right]-\Sigma,
\]
where
\[
\Sigma=
\begin{cases}
\sum_{j=1}^{\lfloor\frac{t-1}{2}\rfloor}
\frac{2}{4^j}\sum_{k\ge0}\binom{2j}{j-\frac{m}{2}-km}
& \text{if $m$ is even,} \\
\sum_{j=1}^{\lfloor\frac{t}{2}\rfloor}
\frac{2}{4^j}\sum_{k\ge0}
\binom{2j-1}{\frac{2j-1-m}{2}-km}
& \text{if $m$ is odd.}
\end{cases}
\]
Here, binomial coefficients $\binom{a}{b}$ with $b$ negative
(or $b=-\infty$) vanish by definition.
\end{theo}
Troili's proof is built on the observation that
\[
E^{I_2(m)}_{S,\ell}(t)=
\sum_{r=0}^{t-1}\bigl[\Prob\bigl(\ell(\pi^{(r)})=0\bigr)
-\Prob\bigl(\ell(\pi^{(r)})=m\bigr)\bigr],
\]
where $\pi^{(r)}$ is a product of $r$ random reflections in $I_2(m)$.

\begin{theo}[Eriksen, Hultman, 2005]\label{th:eriksenhultman}
Let $r,n$ be positive integers not both equal to one.
Then the expected absolute length of a product of $t$ random reflections in
the complex reflection group $G(r,1,n)$ is
\begin{multline*}
E^{G(r,1,n)}_{T,\ell'}(t)=\\
n-\frac{1}{r}\sum_{k=1}^n\frac{1}{k}+
\frac{1}{r}\sum_{p=1}^{n-1}\sum_{q=1}^{\min(p,n-p)}a_{pq}
\left(\frac{r\bigl(\binom{p}{2}+\binom{q-1}{2}-\binom{n-p-q+2}{2}+n\bigr)-n}
{r\binom{n+1}{2}-n}\right)^t \\
+\frac{r-1}{r}\sum_{p=0}^{n-1}\sum_{q=1}^{n-p}b_{pq}
\left(\frac{r\bigl(\binom{p}{2}+\binom{q}{2}-\binom{n-p-q+1}{2}+p\bigr)-n}
{r\binom{n+1}{2}-n}\right)^t,
\end{multline*}
where
\[
a_{pq}=(-1)^{n-p-q+1}\frac{(p-q+1)^2}{(n-q+1)^2(n-p)}\binom{n}{p}
\binom{n-p-1}{q-1}
\]
and
\[
b_{pq}=\frac{(-1)^{n-p-q+1}}{n-p}\binom{n}{p}\binom{n-p-1}{q-1}.
\]
\end{theo}
Note that $A_{n-1}\simeq G(1,1,n)$ and $B_n\simeq G(2,1,n)$.

The proof of Eriksen and Hultman's formula involves group
representations and characters. To see why this is natural, 
think of the symmetric group $A_{n-1}$ where
the absolute length of a permutation is $n$ minus
the number of cycles in the cycle representation of the permutation.
Applying a random transposition either splits a cycle or merges two cycles,
and it is easy to see that during this process we only have to
keep track of the cycle type of the permutation, i.e.~its
conjugacy class.

\section{The dihedral groups}\label{sec:typeI}
\noindent
As we have seen,
computing $E^{I_2(m)}_{S,\ell}(t)$ for the dihedral groups
lead to a rather cumbersome expression in Theorem~\ref{th:troili}.
In this section, we will see that much nicer formulas exist
for $E^{I_2(m)}_{T,\ell}(t)$,
$E^{I_2(m)}_{S,\ell'}(t)$, and $E^{I_2(m)}_{T,\ell'}(t)$.

\begin{theo}\label{th:dihedralTl}
For $m\ge2$ and $t\ge1$, the expected length $E^{I_2(m)}_{T,\ell}(t)$
of a product of $t$ random
reflections in the dihedral group $I_2(m)$ is $m/2$ (for all positive $t$)
if $m$ is even and $\frac{m}{2}-\frac{(-1)^t}{2m}$ if $m$ is odd.
\end{theo}
\begin{proof}
Note that in $I_2(m)$ the set $T$ of reflections is precisely the set
of elements with odd length. Thus, multiplying any element in $T$ by
a uniformly chosen random reflection yields a uniformly chosen random
element in $I_2(m)\setminus T$, and, similarly, multiplying any
element in $I_2(m)\setminus T$ by a uniformly chosen random reflection
yields a uniformly chosen random element in $T$. A simple calculation
reveals that the average length of an element in $T$ is $m/2$ if $m$ is
even and $\frac{m}{2}+\frac{1}{2m}$ if $m$ is odd, and that
the average length of an element in $I_2(m)\setminus T$ is $m/2$ if $m$ is
even and $\frac{m}{2}-\frac{1}{2m}$ if $m$ is odd.
\end{proof}

\begin{theo}
For $m\in\{2,3,\dotsc,\infty\}$ and $t\ge0$, the expected absolute
length $E^{I_2(m)}_{S,\ell'}(t)$
of a product of $t$ random
simple reflections in the dihedral group
$I_2(m)$ is 1 if $t$ is odd and
\[
2-\frac{1}{2^{t-1}}\sum_{\abs{k}\le\lfloor t/2m\rfloor}
\binom{t}{\frac{t}{2}-km}
\]
if $t$ is even. (If $m=\infty$, the sum contains only one term.)
\end{theo}
\begin{proof}
Clearly, in a dihedral group, the absolute length of an element $w$ is
1 if $\ell(w)$ is odd and 2 if $\ell(w)$ is even, unless $w=\id$.
The probability that a product of an even number $t$ of simple
reflections is the identity element is precisely the sum above
divided by $2^t$.
\end{proof}

\begin{theo}
For $m\ge2$ and $t\ge1$, the expected absolute
length $E^{I_2(m)}_{T,\ell'}(t)$
of a product of $t$ random
reflections in the dihedral group
$I_2(m)$ is 1 if $t$ is odd and $2-\frac{2}{m}$ if $t$ is even.
\end{theo}
\begin{proof}
The theorem follows from the same reasoning as in the proof of
Theorem~\ref{th:dihedralTl} and the observation that
the absolute length of any element in $T$ is 1 and
the average absolute length of an element in
$I_2(m)\setminus T$ is $2-\frac{2}{m}$.
\end{proof}

\section{The Weyl groups of type A, B and D}\label{sec:typeABD}
\noindent
In this section we will obtain exact expressions for
$E^W_{T,\ell}(t)$ when $W$ is a Weyl group of type A, B, and D.
The formulas and their proofs turn out to be quite similar for
these three types, but the symmetric groups happen to be a bit less
technical to deal with, so we start with those.

\subsection{The symmetric groups}
\noindent
\begin{theo}\label{th:A}
For $n\ge2$ and $t\ge0$, let $\pi^{(t)}\in A_{n-1}$ be the random
permutation obtained by
applying $t$ random transpositions to the identity permutation.
Then the following holds.
\begin{enumerate}
\item[(a)]
The expected number of inversions in $\pi^{(t)}$ is
\[
E^{A_{n-1}}_{T,\ell}(t)
=\frac{n(n-1)}{4}-\frac{(n+1)(n-1)}{6}\left(1-\frac{2}{n-1}\right)^t
-\frac{(n-1)(n-2)}{12}\left(1-\frac{4}{n-1}\right)^t.
\]
\item[(b)]
For $1\le i<j\le n$, the probability that
$\pi^{(t)}_i>\pi^{(t)}_j$ is an inversion is
\[
\frac{1}{2}-\frac{j-i}{n}\left(1-\frac{2}{n-1}\right)^t
+\left(\frac{j-i}{n}-\frac{1}{2}\right)\left(1-\frac{4}{n-1}\right)^t.
\]
\end{enumerate}
\end{theo}

\begin{proof}
Part (a) of the theorem follows from part (b) by the simple observation that
$\sum_{1\le i<j\le n}(j-i)=n(n^2-1)/6$.

It remains to prove part (b).
For every $1\le i,j\le n$ and $t\ge0$, let
\[
p^{(t)}_{i,j} := {\rm Prob}(\pi^{(t)}_i < \pi^{(t)}_j)
\text{\ \ \ \ \ \ \ and\ \ \ \ \ \ \ }
u^{(t)}_{i,j} := \binom{n}{2}^t p^{(t)}_{i,j}.
\]
Clearly, $u^{(t)}_{i,i}=0$ and for $i\ne j$,
\begin{equation}\label{eq:urec}
u^{(t+1)}_{i, j}=\binom{n}{2}u^{(t)}_{i, j}
+(u^{(t)}_{ j,i}-u^{(t)}_{i, j})
+\sum_{i'\ne j} (u^{(t)}_{i',j}-u^{(t)}_{i,j})
+\sum_{j'\ne i} (u^{(t)}_{i,j'}-u^{(t)}_{i,j}).
\end{equation}
Now, define $v^{(t)}_{i, j}:=u^{(t)}_{i, j}-u^{(t)}_{ j,i}$ for
$1\le i, j\le n$. The above recurrence for $u$ translates to the following
recurrence for $v$.
\[
v^{(t+1)}_{i, j}=\binom{n}{2}v^{(t)}_{i, j}
-2v^{(t)}_{i, j}
+\sum_{i'\ne j} (v^{(t)}_{i',j}-v^{(t)}_{i,j})
+\sum_{j'\ne i} (v^{(t)}_{i,j'}-v^{(t)}_{i,j})
\]
Using that $v^{(t)}_{i,i}=0$ for any $i$, we can write the recurrence
in a more elegant way:
\begin{equation}\label{eq:vrec}
v^{(t+1)}_{i,j}=\frac{n(n-5)}{2}v^{(t)}_{i, j}
+\sum_{i'=1}^n v^{(t)}_{i',j}
+\sum_{j'=1}^n v^{(t)}_{i,j'}.
\end{equation}
Next, we will express the above recurrence in terms of linear operators.
Let $\mathcal{A}_n$ denote the vector space over $\R$ of all real
antisymmetric $n\times n$ matrices, and
define the linear operator $Q$ on $\mathcal{A}_n$
as follows. For any $v\in\mathcal{A}_n$ and for any $1\le i, j\le n$,
let each entry in $Qv$ be its row sum plus its and column sum in $v$,
or formally,
\[
(Qv)_{i, j} = \sum_{i'=1}^n v_{i',j}+\sum_{j'=1}^n v_{i,j'}.
\]
Now, the recurrence relation~\eqref{eq:vrec} can be written
\[
v^{(t+1)}=(Q+xI)v^{(t)},
\]
where $I$ is the identity operator on $\mathcal{A}_n$ and $x=\frac{n(n-5)}{2}$.
The starting point $v^{(0)}$ is the $n\times n$
matrix with zeroes on the main diagonal, ones above it and minus ones below it.
From the simple but crucial observation that $Q^2=nQ$ we deduce that
\[
v^{(t)}=(Q+xI)^t v^{(0)}
=\left[(n+x)^t\frac{Q}{n}+x^t\left(I-\frac{Q}{n}\right)\right]v^{(0)}.
\]
Since $(Qv^{(0)})_{i, j} = 2( j-i)$ we obtain
\[
v^{(t)}_{i, j}=
\frac{2( j-i)}{n}(n+x)^t + \left(1-\frac{2( j-i)}{n}\right)x^t
\]
if $ j>i$.
Dividing this by $\binom{n}{2}^t$ yields
\[
1-2p^{t}_{ j,i}=p^{t}_{i, j}-p^{t}_{ j,i}
=
\frac{2( j-i)}{n}\left(1-\frac{2}{n-1}\right)^t
+ \left(1-\frac{2( j-i)}{n}\right)\left(1-\frac{4}{n-1}\right)^t
\]
which proves part (b) of the theorem.
\end{proof}

A striking feature of the second part of Theorem~\ref{th:A} is
that, for a fixed $t$, the probility that $i<j$ is an inversion
depends only on the difference $j-i$. This fact is not evident from the
definition of the random process, but nevertheless it can be proved
without invoking Theorem~\ref{th:A}.

\begin{prop}
Fix $n\ge2$ and $t\ge0$, and let $\pi\in A_{n-1}$
be a product of $t$ random transpositions. Then,
the probability that $\pi_i>\pi_{j}$ is an inversion
(where $1\le i<j\le n$) is only dependent on $j-i$.
\end{prop}
\begin{proof}
Let $\pi^{+k}$ be the permutation
obtained by adding $k$ ``modulo $n$'' to all entries of $\pi$,
i.e.
\[
\pi^{+k}_i=\begin{cases}
\pi_i+k & \text{if } \pi_i+k\le n, \\
\pi_i+k-n & \text{if } \pi_i+k>n.
\end{cases}
\]
Clearly, for any $1\le i<j\le n$, we have
\[
\Prob(\pi_i>\pi_j)=\Prob(\pi^{+k}_i>\pi^{+k}_j)
+\Prob(\pi_i>n-k)-\Prob(\pi_j>n-k).
\]

Now, fix $i,j,k$ with $1\le i<j<j+k\le n$. It suffices to prove that
\begin{equation}\label{eq:pitrans}
\Prob(\pi_{i+k}>\pi_{j+k})=\Prob(\pi_i>\pi_j).
\end{equation}

By symmetry of the random process, we have
$\Prob(\pi_i>n-k)=\Prob(\pi_j>n-k)$ and hence
$\Prob(\pi_i>\pi_j)=\Prob(\pi^{+k}_i>\pi^{+k}_j)$.
Equation~\eqref{eq:pitrans} now follows from the observation that
$\Prob(\pi_{i+k}>\pi_{j+k})=\Prob(\pi^{+k}_i>\pi^{+k}_j)$.
\end{proof}

\subsection{A definition and a lemma}
\noindent
The situation becomes a little more delicate for Weyl groups of type
B and D, but in a way that is very similar for these types. To
be economical we gather the bit of reasoning that is common for type
B and D in a definition and a lemma.

\begin{defi}
For any positive integer $n$,
let $\mathcal{I}_n:=\{(i,j)\,:\,i,j\in[-n,n]\setminus\{0\},
\,\abs{i}\ne\abs{j}\}$.
\end{defi}

\begin{lemma}\label{lm:BD}
Let $n$ be a positive integer and $x$ a real number.
For every integer $t\ge0$, define a function
$v^{(t)}\,:\,\mathcal{I}_n\rightarrow\R$ by the recurrence
relation
\[
v^{(t+1)}_{i, j} = xv^{(t)}_{i, j}
+\sum_{\abs{i'}\ne\abs{ j}}v^{(t)}_{i', j}
+\sum_{\abs{j'}\ne\abs{i}}v^{(t)}_{i,j'},
\]
together with the initial condition $v^{(0)}_{i, j}=\sgn( j-i)$.
Then these functions are given by
\[
v^{(t)}_{i, j}=
\frac{ j-i-\sgn j+\sgn i}{n-1}(2n-2+x)^t
+\left(\sgn( j-i)-\frac{ j-i-\sgn j+\sgn i}{n-1}\right)x^t.
\]
\end{lemma}

\begin{proof}
Let $\mathcal{D}_n$ denote the real vector space of all functions
$v\,:\,\mathcal{I}_n\rightarrow\R$ such that $v_{ j,i}=-v_{i, j}$ and
$v_{- j,-i}=v_{i, j}$ for any $(i, j)\in\mathcal{I}_n$. (We write
$v_{i, j}$ instead of the more common $v(i, j)$ since we may view
$\mathcal{D}_n$ as the space of $2n\times2n$ matrices that are
antisymmetric about the main diagonal and symmetric about the
antidiagonal, if we disregard the entries on the diagonals.)

Let $Q$ be the linear operator on $\mathcal{D}_n$
defined by
\[
(Qv)_{i, j} =
\sum_{\abs{i'}\ne\abs{ j}}v_{i', j}
+\sum_{\abs{j'}\ne\abs{i}}v_{i,j'}.
\]
for any $v\in\mathcal{D}_n$ and $(i, j)\in\mathcal{I}_n$,

We have
\begin{align*}
&(Q^2v)_{i, j} = \sum_{\abs{i'}\ne\abs{ j}}(Qv)_{i', j}
+\sum_{\abs{j'}\ne\abs{i}}(Qv)_{i,j'} \\
&=
\sum_{\abs{i'}\ne\abs{ j}}
\left(\sum_{\abs{i''}\ne\abs{ j}}v_{i'', j}
+\sum_{\abs{j''}\ne\abs{i'}}v_{i',j''}\right)
+\sum_{\abs{j'}\ne\abs{i}}
\left(\sum_{\abs{i''}\ne\abs{j'}}v_{i'',j'}
+\sum_{\abs{j''}\ne\abs{i}}v_{i,j''}\right)\\
&=
(2n-2)\sum_{\abs{i''}\ne\abs{ j}}v_{i'', j}
+\sum_{(j'',i')\in\mathcal{I}_n}v_{i',j''}
-\sum_{\abs{j''}\ne\abs{ j}}(v_{ j,j''}+v_{- j,j''})\\
&\quad +\sum_{(i'',j')\in\mathcal{I}_n}v_{i'',j'}
-\sum_{\abs{i''}\ne\abs{i}}(v_{i'',i}+v_{i'',-i})
+(2n-2)\sum_{\abs{j''}\ne\abs{i}}v_{i,j''}\\
&\text{(now using the symmetry and antisymmetry properties)}\\
&=(2n-2)\left(\sum_{\abs{i''}\ne\abs{ j}}v_{i'', j}
+\sum_{\abs{j''}\ne\abs{i}}v_{i,j''}\right) \\
&=(2n-2)(Qv)_{i, j}.
\end{align*}
Thus, we have showed that
$Q^2=(2n-2)Q$ which means that
\[
(Q+xI)^t v^{(0)}
=\left[(2n-2+x)^t\frac{Q}{2n-2}
+x^t\left(I-\frac{Q}{2n-2}\right)\right]v^{(0)},
\]
where $I$ is the identity operator on $\mathcal{D}_n$.
The lemma now follows from the observation that
\[
(Qv^{(0)})_{i, j} = \sum_{\abs{i'}\ne\abs{ j}}\sgn( j-i')
+\sum_{\abs{j'}\ne\abs{i}}\sgn(j'-i)
=2\bigl( j-i-\sgn j+\sgn i\bigr).
\]
\end{proof}

\subsection{The hyperoctahedral groups}
\noindent
Recall that $B_n$ is isomorphic to the
subgroup of permutations $\pi$ of the set $[-n,n]\setminus\{0\}$ such that
$\pi_{-i}=-\pi_i$ for all $1\le i\le n$ (see
e.g.~\cite[Ch.~8]{bjornerbrenti}).
In this representation, the set of reflections is
\[
\{(i,j)(-i,-j)\,:\,1\le\abs{i}<j\le n\}\cup\{(i,-i)\,:\,1\le i\le n\}
\]
in cycle notation.
The length of $\pi$ equals its number of \emph{B-inversions}, which are
pairs $(i,j)$ with $i,j\in[-n,n]\setminus\{0\}$ and $j\ge\abs{i}$ such that
$\pi_i>\pi_j$.

\begin{theo}
For $n\ge1$ and $t\ge0$, let $\pi^{(t)}$ be a product of $t$ random
reflections in $B_n$. Then the following holds.
\begin{enumerate}
\item[(a)]
The expected length of $\pi^{(t)}$ is
\[
E^{B_n}_{T,\ell}(t)=
\frac{n^2}{2}-\frac{n(n+1)}{3}\left(1-\frac{2}{n}\right)^t
-\frac{n(n-2)}{6}\left(1-\frac{4}{n}+\frac{2}{n^2}\right)^t.
\]
\item[(b)]
If $i, j\in[-n,n]\setminus\{0\}$ and $j>\abs{i}$,
the probability that $\pi^{(t)}_i>\pi^{(t)}_j$ is
\[
\frac{1}{2}-\frac{ j-i-1+\sgn i}{2(n-1)}
\left(1-\frac{2}{n}\right)^t
+\left(\frac{ j-i-1+\sgn i}{2(n-1)}-\frac{1}{2}\right)
\left(1-\frac{4}{n}+\frac{2}{n^2}\right)^t.
\]
For $1\le i\le n$, the probability that
$\pi^{(t)}_{-i}>\pi^{(t)}_{i}$ is
\[
\frac{1}{2}-\frac{1}{2}\left(1-\frac{2}{n}\right)^t.
\]
\end{enumerate}
\end{theo}

\begin{proof}
Part (a) of the theorem follows from part (b) by the simple observation that
$\sum_{ j>\abs{i}}( j-i)=2n(n^2-1)/3$.

It remains to prove part (b).
For every $i, j\in[-n,n]\setminus\{0\}$ and $t\ge0$, let
\[
p^{(t)}_{i, j} := {\rm Prob}(\pi^{(t)}_i < \pi^{(t)}_ j)
\text{\ \ \ \ \ \ \ and\ \ \ \ \ \ \ }
u^{(t)}_{i, j} := n^{2t} p^{(t)}_{i, j}.
\]
The following recurrence relation holds.
\begin{equation}\label{eq:urecB}
\begin{split}
u^{(t+1)}_{i, j} &= n^2 u^{(t)}_{i, j}
+(u^{(t)}_{ j,i}-u^{(t)}_{i, j})
+(u^{(t)}_{- j,-i}-u^{(t)}_{i, j}) \\
& \quad +\sum_{\abs{i'}\ne\abs{ j}} (u^{(t)}_{i', j}-u^{(t)}_{i, j})
+\sum_{\abs{j'}\ne\abs{i}} (u^{(t)}_{i,j'}-u^{(t)}_{i, j}) \quad\quad
\text{if $\abs{ j}\ne\abs{i}$, and}\\
u^{(t+1)}_{-i,i} &= n^2 u^{(t)}_{-i,i} + (u^{(t)}_{i,-i}-u^{(t)}_{-i,i})
+\sum_{\abs{i'}\ne\abs{i}} (u^{(t)}_{-i',i'}-u^{(t)}_{-i,i}).
\end{split}
\end{equation}
We also have the symmetry property $u^{(t)}_{- j,-i}=u^{(t)}_{i, j}$.

Now, define $v^{(t)}_{i, j}:=u^{(t)}_{i, j}-u^{(t)}_{ j,i}$ for
$i, j\in[-n,n]\setminus\{0\}$. Clearly, $v$ inherits the symmetry
property from $u$ and has also the antisymmetry property
$v^{(t)}_{ j,i} = -v^{(t)}_{i, j}$.
The above recurrence for $u$ is valid also for $v$ and using the
symmetry and antisymmetry properties it can be written
\begin{equation}\label{eq:vrecB}
\begin{split}
v^{(t+1)}_{i, j} &= (n^2-4n+2)v^{(t)}_{i, j}
+\sum_{\abs{i'}\ne\abs{j}}v^{(t)}_{i', j}
+\sum_{\abs{j'}\ne\abs{i}}v^{(t)}_{i,j'}
\quad\quad\text{if $\abs{ j}\ne\abs{i}$, and}\\
v^{(t+1)}_{-i,i} &= n(n-2)v^{(t)}_{-i,i}
\end{split}
\end{equation}
Clearly, this implies that
$v^{(t)}_{i,-i}=(n^2-2n)^t$ for $1\le i\le n$,
and if $ j>\abs{i}$ Lemma~\eqref{lm:BD} yields
\[
v^{(t)}_{i, j}=
\frac{ j-i-1+\sgn i}{n-1}(n^2-2n)^t
+\left(1-\frac{ j-i-1+\sgn i}{n-1}\right)
(n^2-4n+2)^t.
\]

After dividing this by $n^{2t}$ we obtain
\begin{align*}
& 1-2p^{t}_{ j,i}=p^{t}_{i, j}-p^{t}_{ j,i} \\
&=
\frac{ j-i-1+\sgn i}{n-1}\left(1-\frac{2}{n}\right)^t
+\left(1-\frac{ j-i-1+\sgn i}{n-1}\right)
\left(1-\frac{4}{n}+\frac{2}{n^2}\right)^t
\end{align*}
if 
$ j>\abs{i}$ and
\[
1-2p^{t}_{-i,i}=\left(1-\frac{2}{n}\right)^t
\]
for $1\le i\le n]$. This proves part (b) of the theorem.
\end{proof}

\subsection{The groups of type D}
\noindent
Recall that $D_n$ is isomorphic to the
subgroup of permutations $\pi$ of the set $[-n,n]\setminus\{0\}$ such that
$\pi_{-i}=-\pi_i$ for all $1\le i\le n$ and there is an even number
of $i\in\{1,2,\dotsc,n\}$ such that $\pi_i<0$ (again see
e.g.~\cite[Ch.~8]{bjornerbrenti}).
In this representation, the set of reflections is
\[
\{(i,j)(-i,-j)\,:\,1\le \abs{i}<j\le n\}
\]
in cycle notation.
The length of $\pi$ equals its number of \emph{D-inversions}, which are
pairs $(i,j)$ with $i,j\in[-n,n]\setminus\{0\}$ and $j>\abs{i}$ such that
$\pi_i>\pi_j$.

\begin{theo}
For $n\ge1$ and $t\ge0$, let $\pi^{(t)}$ be a product of
$t$ random reflections in $D_n$. Then the following holds.
\begin{enumerate}
\item[(a)]
The expected length of $\pi^{(t)}$ is
\[
E^{D_n}_{T,\ell}(t)=
\frac{n(n-1)}{2}-\frac{n(2n-1)}{6}\left(1-\frac{2}{n}\right)^t
-\frac{n(n-2)}{6}\left(1-\frac{4}{n}\right)^t.
\]
\item[(b)]
If $i, j\in[-n,n]\setminus\{0\}$ and $ j>\abs{i}$,
the probability that $\pi^{(t)}_i>\pi^{(t)}_{ j}$ is
\[
\frac{1}{2}-\frac{ j-i-1+\sgn i}{2(n-1)}
\left(1-\frac{2}{n}\right)^t
+\left(\frac{ j-i-1+\sgn i}{2(n-1)}-\frac{1}{2}\right)
\left(1-\frac{4}{n}\right)^t.
\]
\end{enumerate}
\end{theo}

\begin{proof}
Part (a) of the theorem follows from part (b) by the simple observation that
$\sum_{ j>\abs{i}}( j-i)=2n(n^2-1)/3$.

It remains to prove part (b).
For every $t\ge0$ and
$(i, j)\in\mathcal{I}_n$, let
\[
p^{(t)}_{i, j} := {\rm Prob}(\pi^{(t)}_i < \pi^{(t)}_ j)
\text{\ \ \ \ \ \ \ and\ \ \ \ \ \ \ }
u^{(t)}_{i, j} := (n^2-n)^t p^{(t)}_{i, j}.
\]
The following recurrence relation holds:
\begin{equation}\label{eq:urecD}
\begin{split}
u^{(t+1)}_{i, j} &= (n^2-n) u^{(t)}_{i, j}
+(u^{(t)}_{ j,i}-u^{(t)}_{i, j})
+(u^{(t)}_{- j,-i}-u^{(t)}_{i, j}) \\
& \quad +\sum_{\abs{i'}\ne\abs{i},\abs{ j}} 
(u^{(t)}_{i', j}-u^{(t)}_{i, j})
+\sum_{\abs{j'}\ne\abs{i},\abs{ j}}
(u^{(t)}_{i,j'}-u^{(t)}_{i, j}) \quad\quad
\end{split}
\end{equation}
We also have the symmetry property $u^{(t)}_{- j,-i}=u^{(t)}_{i, j}$.

Now, define $v^{(t)}_{i, j}:=u^{(t)}_{i, j}-u^{(t)}_{ j,i}$ for
$(i, j)\in\mathcal{I}$.
Clearly, $v^{(t)}$ inherits the symmetry
property from $u^{(t)}$ and has also the antisymmetry property
$v^{(t)}_{ j,i} = -v^{(t)}_{i, j}$.
The above recurrence for $u$ is valid also for $v$ and using the
symmetry and antisymmetry properties it can be written
\begin{equation}\label{eq:vrecD}
v^{(t+1)}_{i, j} = (n^2-5n+4)v^{(t)}_{i, j}
+\sum_{\abs{i'}\ne\abs{ j}}v^{(t)}_{i', j}
+\sum_{\abs{j'}\ne\abs{i}}v^{(t)}_{i,j'}.
\end{equation}
By Lemma~\ref{lm:BD} we obtain, for $j>\abs{i}$,
\[
v^{(t)}_{i, j}=
\frac{ j-i-1+\sgn i}{n-1}(n^2-3n+2)^t
+\left(1-\frac{ j-i-1+\sgn i}{n-1}\right)(n^2-5n+4)^t.
\]
Dividing by $(n^2-n)^t$ yields
\begin{align*}
& 1-2p^{t}_{ j,i}=p^{t}_{i, j}-p^{t}_{ j,i} \\
&=
\frac{ j-i-1+\sgn i}{n-1}\left(1-\frac{2}{n}\right)^t
+\left(1-\frac{ j-i-1+\sgn i}{n-1}\right)
\left(1-\frac{4}{n}\right)^t
\end{align*}
if 
$ j>\abs{i}$. This proves part (b) of the theorem.
\end{proof}

\section{Conclusion and ideas for future research}\label{sec:future}
\noindent
We can sum up the current state of knowledge by the following table,
listing the groups $W$ for which we have exact expressions
for $E^W_{R,\varphi}(t)$.
\vskip5mm
\begin{center}
\begin{tabular}{l||c|c}
& $R=$ simple reflections & $R=$ all reflections \\
\hline\hline
$\varphi=$ length & $A_n$, $I_2(m)$, $I_2(\infty)$
& $A_n$, $B_n$, $I_2(m)$, $G(r,1,n)$ \\
\hline
$\varphi=$ absolute length & $I_2(m)$, $I_2(\infty)$
& $A_n$, $B_n$, $D_n$, $I_2(m)$
\end{tabular}
\end{center}
\vskip5mm
Perhaps the two most striking things with this table are that
$B_n$ is missing from the upper left square and that
not even $A_n$ appears in the lower left square!

As Troili has pointed out~\cite{troili}, the problem of
computing $E^{B_n}_{S,\ell}(t)$ can be reduced to a heat flow process
similar to that of the symmetric group, but it is not clear how to
analyze it further.

Computing $E^{A_n}_{S,\ell'}(t)$ is equivalent to computing
the expected number of cycles (in the cycle representation of the
permutation) after applying $t$ random adjacent transpositions to
the identity permutation. For now,
we have no idea how to attack this natural problem.
Note that it is not enough to
keep track of the conjugacy class of the permutation as Eriksen and
Hultman did in the proof of Theorem~\ref{th:eriksenhultman}.

It is also worth noting that $E^W_{S,\ell}(t)$ and $E^W_{S,\ell'}(t)$
is well-defined for any finitely generated Coxeter system. (Indeed,
Troili made an effort to compute $E^{\tilde{A}_n}_{S,\ell}(t)$
in~\cite{troili}, but unfortunately with an erraneous heat flow process.)

In addition to the length and the absolute length,
there is at least one more ``length'' function that comes naturally
with any Coxeter system $(W,S)$, namely the
\emph{descent number} defined by
\[
d(w)=\#\{s\in S\,:\,\ell(ws)<\ell(w)\}.
\]
(In fact, the ordinary length function is also a kind of descent number
by the equality $\ell(w)=\#\{t\in T\,:\,\ell(wt)<\ell(w)\}$, see
e.g.~\cite[Corollary~1.4.5]{bjornerbrenti}.)
What is $E^W_{S,d}(t)$ and $E^W_{T,d}(t)$ for a (in the second case, finite)
Coxeter system $(W,S)$?

\section{Acknowledgement}
\noindent
This work was performed at KTH in Stockholm and was supported
by a grant from the Swedish Research Council (621-2009-6090).

\bibliographystyle{abbrv}
\bibliography{expected}

\end{document}